\documentclass[12pt,twoside]{amsart}
\usepackage[all]{xy}
        \usepackage {amssymb,latexsym,amsthm,amsmath,mathtools,mathrsfs}
        \usepackage{enumitem,color}
        
\usepackage{array}

\allowdisplaybreaks[4]
\usepackage{hyperref}
\usepackage[capitalize,nameinlink]{cleveref} 
\usepackage{xcolor} 
\hypersetup{
    colorlinks,
    linkcolor={red!80!black},
    citecolor={green!80!black},
    urlcolor={blue!80!black}
}
\usepackage[maxbibnames=99]{biblatex} 
\usepackage{mathtools}

\long\def\symbolfootnote[#1]#2{\begingroup%
\def\thefootnote{\fnsymbol{footnote}}\footnote[#1]{#2}\endgroup}

\newcommand{\tr}{\ensuremath{{}^t\!}}

\newcommand{\F}{\mathbb{F}}

\newcommand{\GL}{\mathrm{GL}}

\newcommand{\Sp}{\mathrm{Sp}}
\newcommand{\pa}{\Lambda}
\newcommand{\e}{\epsilon}
\newcommand{\fqn}[1]{\mathbb{F}_{q^{#1}}}
\newcommand{\ssp}{\Lambda}
\newcommand{\fq}{\ensuremath{\mathbb{F}_q}}
\newcommand{\On}[1]{\mathrm{O}^{#1}}
\newcommand{\ep}{\epsilon}
\newcommand{\U}{\mathrm{U}}

\newcommand{\p}{\mathbf{P}}

\makeatletter
\def\imod#1{\allowbreak\mkern10mu({\operator@font mod}\,\,#1)}
\makeatother

\newtheorem{theorem}{Theorem}[section]
\newtheorem{lemma}[theorem]{Lemma}
\newtheorem{corollary}[theorem]{Corollary}
\newtheorem{proposition}[theorem]{Proposition}
\newtheorem*{theorem*}{Theorem}
\theoremstyle{definition}

\numberwithin{equation}{section}

\newcommand{\ignore}[1]{}

\newcommand{\mynote}[1]{}
\newcommand{\id}{\mathrm{Id}}
\newcommand{\Z}{\mathbb Z}
\newcommand{\G}{\mathrm G}
\newcommand{\ord}{\mathrm{ord}}
\newcommand{\D}{\widetilde{D}}
\usepackage{parskip}
\begin{document}
\setcounter{section}{0}
\title[Roots of identity]{Roots of identity in finite groups of Lie type}
\author{Saikat Panja}
\email{panjasaikat300@gmail.com}
\address{Harish-Chandra Research Institute- Main Building, Chhatnag Rd, Jhusi, Uttar Pradesh 211019, India}
\thanks{The author is supported by a PDF-Math fellowship from the Harish-Chandra Research Institute}
\date{\today}
\subjclass[2020]{20G40, 20D06}
\keywords{word maps, finite groups of Lie type, $M$-th root}
\begin{abstract}
    Given an integer $M\geq 2$, we deploy the generating function techniques to compute the number of $M$-th roots of identity in some of the well-known finite groups of Lie type, more precisely for finite general linear groups, symplectic groups, orthogonal groups of all types and unitary groups over finite fields of odd characteristics.
\end{abstract}
\maketitle
\section{Introduction}\label{sec:intro}
A significant amount of modern mathematics is centred around finding solutions 
to equations in objects belonging to different categories, for example, the category of
groups, the category of associative algebras, the category of Lie algebras \emph{etc.}. In the world of groups, these are known as \emph{word problems}. It has come to the consideration of many great mathematicians of the current century, especially after the settlement of the long-standing
Ore's conjecture (which states that every element of a finite non-abelian simple 
group is a commutator). Let us recall the word problem in the context of group theory in brief. Fix a free group $\mathscr{F}_n$ on $n$ generators. Then for any word $\omega\in \mathscr{F}_n$ and a group $G$, we get a map
$\widetilde{\omega}:G^n\longrightarrow G$, by means of evaluations. Several astonishing results have been established regarding these maps. A typical \emph{word problem} asks if the map induced by the given word is surjective on the group or not. The other related problem i,e, \emph{Waring-type problem} asks for the
existence of $N(\omega)\in \mathbb{N}$ such that $\langle \widetilde{\omega}(G)\rangle=\widetilde{\omega}(G)^{N(\omega)}$.
The most surprising result in this direction states that, for a given word 
$\omega$, there exists $N_\omega\in\mathbb{N}$ such that for all finite non-abelian simple groups of size greater or equal to $N_w$, we have $N(\omega)=3$, see
the great work \cite{Sha09Waring3} by Shalev. Later he together with Larsen and Tiep improved the result by proving $N(\omega)=2$, in \cite{LarsenShalev09} and 
\cite{LShTi11Waring2}.
Unfortunately the number $N(\omega)$ can not be reduced further. An easy example
arises out of the power map, which is induced by the word $\ omega=x^M\in\mathscr{F}_1$, for some integer $M\geq 2$. The Waring-type problems have also been studied in 
case of Lie groups and Chevalley groups in \cite{HuiLarsenShalev15}, unipotent algebraic groups 
in \cite{Larsen19unipotent}, residually finite groups in \cite{LarsenShalev18residual}, $p$-adic and Adelic groups in \cite{AvniShalev13}
 \emph{et cetera}.
For a survey of these results and further problems in
the context of group theory, we refer the reader to the excellent survey due
to Shalev \cite{Sha09Waring3} and the references therein.

There have been studies to find the number of elements which are images of the map $x^M$ in classical groups. This was done for the finite general linear groups in \cite{KuSi22}, for the finite symplectic and orthogonal 
groups in \cite{PaSi22} and for finite unitary groups in \cite{panja2023powers}.
The exceptional cases remain open. After knowing that an equation has a solution
in the concerned object, it is desirable to ask how many solutions exist in this case. Some authors also call them as the problem of finding fibres of a word map.
The enumeration of the solutions of $x^p=1$ 
has been carried out for symmetric and alternating groups
in many works like \cite{ChowlaHerstein52}, \cite{NiemeyerPraeger07}, \cite{KodaSato15}, \cite{Ishihara01} \emph{et cetera}. The problem of finding 
fibres of $x^M=1$ in some of the finite groups of Lie type has been studied by several authors, for example, proportions of elements in finite groups of Lie type have been estimated in \cite{NiePra10fnitegroups}, 
 algorithms to identify abundant $p$-singular elements in finite classical groups
have been studied in \cite{NiPr12algorithm}, Proportions of {$p$}-regular elements in finite classical groups have been studied in \cite{BabaiGuestPrWilson13}, abundant {$p$}-singular elements in finite classical groups have been studied in \cite{NiPoPr14} \emph{et ceetra}. In this 
paper we derive the generating functions 
for the probability of an element satisfying $x^M=1$, where $M$ is an integer greater than $2$. We call these elements to
be
\emph{roots of identity}. Note that if
an element is a root of identity, then so are all its conjugates. Thus we first find the
conjugacy class representatives which are $M$-th root of identity. In this regard, we brush through a brief description of the conjugacy classes in 
\cref{sec:conjcent}. Hereafter in \cref{sec:semsim-uni}, we first compute the 
roots of identity which are either semisimple or unipotent. Finally using the 
Jordan decomposition of an element, in \cref{sec:gen-fun}, we derive the generating functions
 for the number of $M$-th roots of identity in finite classical groups of Lie type; more precisely in the general linear group, the symplectic group, the orthogonal group and the unitary group. Finally in \cref{sec:calc-prime}, we derive an exact formula for the probability of a matrix to be $M$-th root of identity in the case of the general linear group and the symplectic group, where $M$ is an odd non-defining prime and the defining prime satisfies $q\equiv -1\pmod{M}$.
\subsection*{Running assumption} {\color{blue}Throughout the article it will be assumed that the finite fields under consideration are of odd characteristics.}
\subsection*{Notation} The set of units in $\Z/n\Z$ will be denoted by $\Z/n\Z^\times$. By the notation $e(n)$, we will always denote the order of $q$ in the multiplicative group
$\Z/n\Z^\times$. The Euler's totient function will be denoted by $\phi(n)$. We denote the set of irreducible monic polynomials over $\fq$ by $\Phi$. Let us also
fix the notation for the set of all partitions of $n$ to be $\pa^n$. By the symbol 
$\left(\dfrac{u}{q}\right)_i$, we denote the quantity $\left(1-\dfrac{u}{q}\right)\left(1-\dfrac{u}{q^2}\right)\ldots \left(1-\dfrac{u}{q^i}\right)$. For a partition $\lambda\vdash n$, the number of occurrences of $i$ will be denoted as 
$m_i(\lambda)$. The transpose of a partition $\lambda$ will be denoted as $\lambda'$. The number of partitions of a positive integer $n$ will be denoted 
as $\p(n)$ and those with lesser than or equal to equal to $m$ many (not necessarily distinct) parts will be denoted as 
$\p_m(n)$.

\section{The groups under consideration}\label{sec:conjcent}
In this section, we briefly recall the groups. Note that if an element is an $M$-th root of identity, then so are its conjugates. 
Hence we embark to find first the conjugacy classes which are $M$-th root of identity. This makes it necessary for us to know about the conjugacy classes and the centralizer of an element. We will follow the treatment of \cite{Green55},
\cite{Macdonald81}, \cite{wall63} and, \cite{Enn62}. For the case of
the general linear group, it is easy to find the conjugacy representative for
the conjugacy classes, but for other concerned groups, they are a little tricky. 
One may look into \cite{ta1}, \cite{ta2} and, \cite{ta3}, to handle the other cases. Without further delay let us start with the subsection where we talk about the general linear group.
\subsection{General linear group} Given a $n$-dimensional vector space $V$ over $\fq$, the set of all automorphisms of $V$ is denoted by $\GL(V)$. After
fixing an appropriate basis, $\GL(V)$ can be identified with $\GL_n(q)$. 
The conjugacy classes of $\GL_n(q)$ is determined by functions 
$\lambda:\Phi\longrightarrow\Lambda$ by $f\mapsto\lambda_f$ such that $\sum\limits_{f}\deg(f)\cdot|\lambda_f|=n$. The combinatorial data attached via this correspondence, for an element $\alpha\in \GL_n(q)$ will be denoted by $\Delta(\alpha)$. The centralizer size 
is given by
\begin{align*}
    \prod\limits_{\varphi\in\Delta(\alpha)}\left(q^{\deg \varphi\cdot \sum\limits_{i}(\lambda_{i,\varphi}')^2}\prod_{i\geq 1}\left(\dfrac{1}{q^{\deg\varphi}}\right)_{m_i(\lambda_\varphi)}\right).
\end{align*}
This quantity sometimes can be shortened using a general expression, which we recall from \cite[Theorem 2]{Ful99CycInd} below; 
\begin{align*}
    c_{\GL,\varphi, q}(\lambda)=\left(q^{\deg \varphi\cdot \sum\limits_{i}(\lambda_{i}')^2}\prod_{i\geq 1}\left(\dfrac{1}{q^{\deg\varphi}}\right)_{m_i(\lambda_\varphi)}\right).
\end{align*}
\subsection{Symplectic and orthogonal groups}
\subsubsection{Symplectic group}
Let $V$ be a vector space of dimension $2n$ over $\mathbb F_q$. There is a unique non-degenerate alternating bilinear form on $V$. We consider the form given by 
$$\left<(x_i)_{i=1}^{2n}, (y_j)_{j=1}^{2n} \right> = \sum_{j=1}^{n}x_jy_{2n+1-j}-\sum_{i=0 } ^{n-1}x_{2n-i}y_{i+1}.$$ 
The symplectic group is the subgroup of $\GL(V)$ consisting of those elements which preserve this alternating form on $V$. By fixing an appropriate basis, the matrix of the form is $J= \begin{pmatrix} 0 & \Pi_n\\ -\Pi_n & 0 \end{pmatrix}$ where 
$\Pi_n=\begin{pmatrix} 0 & 0 & \cdots & 0 & 1\\
0 & 0 & \cdots & 1 & 0\\ \vdots & \vdots & \reflectbox{$\ddots$} & \vdots & \vdots\\
1 & 0 & \cdots & 0 & 0 \end{pmatrix}$ and 
$$\Sp_{2n}(q)=\{A \in \GL_{2n}(q) \mid \tr AJA=J\}.$$ 
Since all alternating forms are equivalent over $\mathbb F_q$, the symplectic groups obtained with respect to different forms are conjugate within $\GL_{2n}( q)$. 
\subsubsection{Orthgonal Groups}
Let $V$ be an $m$-dimensional vector space over a finite field $\mathbb F_q$. Then there are at most two non-equivalent non-degenerate quadratic forms on $V$. The orthogonal group consists of elements of $\GL(V)$ which preserve a non-degenerate quadratic form $Q$. 

When $m=2n$ for some $n\geq 1$, up to equivalence there are two such forms denoted as $Q^+$ and $Q^-$. These are as follows. Fix $a\in \mathbb F_q$ such that $t^2 + t + a \in \fq[t]$ is irreducible. Then the two non-equivalent forms are given by
\begin{enumerate}
\item $Q^+(x_1, \ldots, x_m) = x_{1}x_{2} + x_3x_4+\cdots + x_{2n-1}x_{2n}$, and
\item $Q^-(x_1,\ldots, x_m)= x_1^2 + x_1x_2+ax_2^2 + x_3x_4+\cdots + x_{2n-1}x_{2n}$.
\end{enumerate}
The orthogonal group preserving $Q^+$ will be denoted as $\On{+}_m(q)$, and the orthogonal group preserving $Q^-$ will be denoted as $\On{-}_m(q)$.

When $m=2n+1$, for $q$ even there is only one (up to equivalence) quadratic form, namely 
$Q(x_1, \ldots, x_m)=x_1^2+\sum\limits_{i=1}^n x_{2i}x_{2i+1}$ and hence there is only one (up to conjugacy) orthogonal group. If $q$ is odd, then up to equivalence, there are two non-degenerate quadratic forms. But, these two forms give isomorphic orthogonal groups. We take $Q(x_1, \ldots, x_m)= x_1^2+\cdots + x_m^2$. Thus, in case $m=2n+1$, up to conjugacy, we have only one orthogonal group. This will be denoted as $\On{}(m, q)$. 

As it is common in literature, we will use the notation $\On{\epsilon}_m(q)$ to denote any of the orthogonal group above where $\epsilon \in \{~, +, -\}$. With respect to an appropriate basis, we will fix the matrices of the symmetric bilinear forms (associated to the quadratic forms $Q^{\epsilon}$) as follows: 
$$J_{0}=\begin{pmatrix} 0 & 0 & \Pi_n \\ 0 & \alpha & 0\\ \Pi_n & 0 & 0 
\end{pmatrix},  J_{+}=\begin{pmatrix} 0& \Pi_n\\ \Pi_n & 0
\end{pmatrix}, J_{-}=\begin{pmatrix} 0 & 0 & 0 & \Pi_{n-1}\\
0 & 1 & 0 & 0 \\ 0 & 0 & -\delta & 0 \\ \Pi_{n-1}& 0 & 0 & 0
\end{pmatrix}$$ 
where $\alpha\in\fq^\times,\delta\in\fq\setminus\fq^2$, and $\Pi_l =\begin{pmatrix}
0 & 0 & \cdots & 0 & 1\\ 0 & 0 & \cdots & 1 & 0\\
\vdots & \vdots & \reflectbox{$\ddots$} & \vdots & \vdots\\
1 & 0 & \cdots & 0 & 0 \end{pmatrix}$, an $l\times l$ matrix. Then, the orthogonal group in matrix form is 
$$\On{\e}_m( q) = \{A \in \GL_m(q)\mid \tr{A}J_{\e}A=J_{\e}\}.$$
Adapting the notations of \cite{FuNePr05book}, we define the type of an orthogonal space as follows.

\emph{The type of an orthogonal space} $(V,Q)$ of dimension $m$ is
\begin{align*}
\tau(V)=\begin{cases}
1 & \text{if } m \text{ is odd, }q\equiv 1\imod 4, Q\sim\sum x_i^2,\\
-1 & \text{if } m \text{ is odd, }q \equiv 1\imod 4, Q\sim b\sum x_i^2,\\
\iota^m & \text{if } m \text{ is odd, }q\equiv 3\imod 4, Q\sim\sum x_i^2,\\
(-\iota)^m & \text{if } m \text{ is odd, }q\equiv 3\imod 4, Q\sim b\sum x_i^2,
\end{cases}
\end{align*}
where $\iota\in\mathbb{C}$ satisfies $\iota^2=-1$, and $b\in\fq\setminus\fq^2$.
 More generally, when $V$ is orthogonal direct sum $V_1\oplus V_2 \oplus \cdots \oplus V_l$, the type is defined by $\tau(V)=\prod\limits_{i=1}^l\tau(V_i)$.


Now we describe the conjugacy classes and the centralizer of an element in 
finite symplectic and orthogonal groups. This is the classic work of \cite{wall63}.We recall briefly the results therein, which will be used further.
A \textbf{symplectic signed partition} is a partition of a number $k$, such that the odd parts 
have even multiplicity and even parts have a sign associated with it. The set
of all symplectic signed partitions will be denoted as $\ssp_{\Sp}$.
An \textbf{orthogonal signed partition} is a partition of a number $k$, such that all even parts have even multiplicity, and all odd parts have a sign associated with it. The set
of all orthogonal signed partitions will be denoted as $\ssp_{\On{}}$.
The \textit{dual} of a monic degree $r$ polynomial $f(x)\in k[x]$ satisfying $f(0)\neq 0$, is the polynomial given by $f^*(x)=f(0)^{-1}x^rf(x^{-1})$. 

According to \cite{wall63}, \cite{shinoda80}, the conjugacy classes of $\Sp_{2n}(q)$ are parameterized by the functions 
$\lambda:\Phi\rightarrow\pa^{2n}\cup\ssp_{\Sp}^{2n}$, where $\Phi$ denotes the set of all
 monic,  non-constant, irreducible polynomials, $\pa^{2n}$ is
 the set of all partitions 
of $1\leq k\leq 2n$ and $\ssp_{\Sp}^{2n}$ is the set of all symplectic
 signed partitions 
of $1\leq k\leq 2n$. Such a $\lambda$ represent a conjugacy class of $\Sp$
if and only if  (a) $\lambda(u)=0$, (b)
 $\lambda_{\varphi^*}=\lambda_\varphi$, (c)
 $\lambda_\varphi\in\ssp^n_{\Sp}$ iff $\varphi=u\pm 1$ (we distinguish this $\lambda$, by denoting it $\lambda^\pm$) and, (d)
 $\displaystyle\sum_{\varphi}|\lambda_\varphi|\textup{deg}(\varphi)=2n$.
Also from \cite{wall63}, \cite{shinoda80}, we find that a similar kind of statement is true for the groups $\On{\epsilon}_n(q)$. The conjugacy classes of $\On{\epsilon}_n(q)$ are parameterized by the functions 
$\lambda:\Phi\rightarrow\pa^{n}\cup\ssp_{\On{}{}}^{n}$, where $\Phi$ denotes the set of all
 monic,  non-constant, irreducible polynomials, $\pa^{n}$ is
 the set of all partitions 
of $1\leq k\leq n$ and $\ssp_{\On{}}^{n}$ is the set of all symplectic
 signed partitions 
of $1\leq k\leq n$. Such a $\lambda$ represent a conjugacy class of $\On{\epsilon}_{n}(q)$ for $\epsilon=\pm$,
if and only if (a)
 $\lambda(x)=0$, (b)
 $\lambda_{\varphi^*}=\lambda_\varphi$, (c)
 $\lambda_\varphi\in\ssp^n_{\On{}}$ iff $\varphi=u\pm 1$ (we distinguish this $\lambda$, by denoting it $\lambda^\pm$) and, (d)
 $\displaystyle\sum_{\varphi}|\lambda_\varphi|\textup{deg}(\varphi)=n$.
Class representative corresponding to given data can be found in \cite{ta1}, \cite{ta2}, \cite{glo} 
and we will mention them whenever needed. We mention the
following results about the conjugacy class size (and hence the size of the centraliser)
 of elements corresponding to given data, which can be found in \cite{wall63}.
\begin{lemma}\cite[pp. 36]{wall63}\label{lem:centralizer-sizes-symp}
Let $X\in\Sp_{2n}(q)$ be a matrix corresponding to the data $\Delta(X)=\{(\phi,\mu_\phi)
:\phi\in\Phi_X\subset \Phi\}$. Then the conjugacy class of
$X$ in $\Sp_{2n}(q)$ is of size $\dfrac{|\Sp_{2n}(q)|}{\prod\limits_{\phi}B(\phi)}$
where $B(\phi)$ and $A(\phi^\mu)$ are defined as follows
\begin{align*}
A(\phi^\mu)=\begin{cases}
|\U_{m_\mu}(Q)| & \text{if } \phi(x)=\phi^*(x)\neq x\pm 1\\
|\GL_{m_\mu}(Q)|^\frac{1}{2}& \text{if } \phi\neq\phi^*\\
|\Sp_{m_\mu}(q)| &\text{if }\phi(x)=x\pm 1,~\mu\text{ odd}\\
|q^{\frac{1}{2}m_\mu}\On{\epsilon}_{m_\mu}(q)|&\text{if }\phi(x)=x\pm 1,~\mu\text{ even}
\end{cases},
\end{align*}
where $\e$ gets determined by the sign of the corresponding partition, 
$Q=q^{|\phi|}$, $m_\mu=m(\phi^\mu)$ and
\begin{align*}
B(\phi) = Q^{\sum\limits_{\mu<\nu}\mu m_\mu m_\nu+
\frac{1}{2}\sum\limits_{\mu}(\mu-1)m_\mu^2}\prod\limits_{\mu}A(\phi^\mu).
\end{align*}
\end{lemma}
\begin{lemma}\cite[pp. 39]{wall63}\label{centra-size-on}
Let $X\in\On{\e}_n(q)$ be a matrix corresponding to the data $\Delta(X)=\{(\phi,\mu_\phi)
:\phi\in\Phi_X\subset \Phi\}$. Then the conjugacy class of
$X$ in $\On{\e}_n(q)$ is of size $\dfrac{|\On{\e}_n(q)|}{\prod\limits_{\phi}B(\phi)}$
where $B(\phi)$ and $A(\phi^\mu)$ are defined as before, except when $\phi(x)=x\pm 1$,
\begin{align*}
A(\phi^\mu)=\begin{cases}
|\On{\epsilon'}_{m_\mu}(q)|&\text{if }\mu\text{ odd}\\
q^{-\frac{1}{2}m_{\mu}}|\Sp_{m_\mu}(q)|&\text{if }\mu\text{ even}
\end{cases},
\end{align*}
where $\e'$ in $\On{\e'}_{m_\mu}(q)$  gets determined by the 
corresponding sign of the part, of the partition.
\end{lemma}
{\color{blue}The quantity $B(\phi)$ will also be denoted as $c_{\Sp,\phi,q}(\lambda^{\pm})$ in case of symplectic groups and as $c_{\On{\epsilon},\phi, q}(\lambda^\pm)$ in case of orthogonal groups when $\phi=x\pm 1$.}
\subsection{Unitary groups}
Consider the field $k=\mathbb F_{q^2}$, and the involution $\sigma \colon \fqn{2}\longrightarrow\fqn{2}$ defined as $\sigma(a):=\overline{a}=a^q$ of the field. We further get an automorphism of
$\fqn{2}[t]$, the polynomial ring over $\fqn{2}$, by action on the coefficients of the polynomials. The image of
$f\in\fqn{2}[t]$ will be denoted as $\overline{f}$. Consider the Hermitian form given by the matrix
\[
\Pi_n=
\begin{pmatrix}
0 & 0 & \ldots & 0 & 1 \\
0 & 0 & \ldots & 1 & 0 \\
\vdots & \vdots & \ddots & \vdots & \vdots \\
0 & 1 & \ldots & 0 & 0 \\
1 & 0 & \ldots & 0 & 0 
\end{pmatrix}.
\]
Then we identify the unitary group (corresponding to the above Hermitian form) with the set
\[
\U_n(\fqn{2})=\left\{A\in\GL_n(\fqn{2})\mid  A\Pi_n\overline{A}^t=\Lambda_n\right\}.
\]
Since all non-degenerate Hermitian forms on an $n$-dimensional vector space are unitarily congruent to the Hermitian form given by the above matrix,
it is evident that all the unitary groups arising are conjugate with each other inside
$\GL_n(q^2)$ (see theorem $10.3$ of \cite{Gr02}). From the work of Wall, it is known that two elements $g_1, g_2 \in\U_n(q^2)$ are conjugate to each other if and only if they are conjugate in $\GL_n(q^2)$.
These conjugacy classes are parameterized by special polynomials and partitions.
We call a polynomial $f$ of degree $d$ to be \emph{$\sim$-symmetric} if $\widetilde{f}=f$ where $\widetilde{f}(t) = \overline{f(0)}^{-1} t^d \bar f(t^{-1})$. A polynomial without a proper $\sim$-symmetric factor will be called a \emph{$\sim$-irrdeucible} polynomial. Note that a $\sim$-irreducible polynomial can be reducible in the usual sense.
Denote by $\widetilde{\Phi}$ to be the set of all $\sim$-irreducible polynomials. 
Let $\mathcal{P}_{n}$ denote the set of all partitions of numbers $\leq n$. Then a conjugacy class in $\U(n,\fqn{2})$ is determined by a function $\lambda\colon \widetilde{\Phi} \longrightarrow \mathcal{P}_n$ which satisfies
(a) $\lambda(f)=\lambda(\widetilde{f})$ and, (b) $\sum\limits_{\phi}|\lambda(\phi)|\deg\phi=n$. The centralizer six=ze can be found in \cite{wall63} and we are not quoting it here.

Using these information,
we determine when a matrix $A$ is an $M$-th root of identity. 
This information is further used in subsequent sections to determine the 
desired generating functions, with the help of the concept of 
central join of two matrices, following \cite{ta1}, \cite{ta2}, \cite{ta3}.

\section{Semisimple \& unipotent roots of identity}\label{sec:semsim-uni}
\subsection{Semisimple classes}
Note that an element in $\G(q)$ is semisimple if and only if it has order coprime to $q$. In this 
subsection, we classify all the semisimple elements, which give $\id$, when raised to the power $M$.
In this subsection, we will assume that $(M,q)=1$. Let $f$ denote a product of distinct irreducible polynomials, and $A$ be the 
semisimple matrix with $f$ as its characteristic polynomial. Then if $A^M=\id$, then $f$ must have irreducible factors coming from irreducible factors of
$x^M-1\in\fq[x]$. We recall some results and definitions from \cite{RuNi97FiniteField}.
 For a positive integer $n$, the splitting field of $x^n-1$ over a field $K$ is called the $n$-th 
 \emph{cyclotomic field} over $K$ and denoted by $K^{(n)}$.
 The roots of $x^n-1$ in $K^{(n)}$ are called the $n$-th roots of unity over $K$ and the set of all
 these roots is denoted by $E^{(n)}$.
\begin{lemma}\cite[Theorem 2.42]{RuNi97FiniteField}\label{lem:cyclotomic-field}
Let $n$ be a positive integer and $K$ a field of characteristic $p$. If $(p,n)=1$, then $E^{(n)}$ is a cyclic 
group of order $n$ with respect to multiplication in $K^{(n)}$.
\end{lemma}
A generator of the cyclic group is called \emph{primitive $n$-th root of unity over $K$}. If $(n,p)=1$ 
and $\zeta$ is a primitive $n$-th root of unity over $K$, the polynomial
\begin{align*}
    Q_n(x)=\prod\limits_{\substack{s=1\\(s,n)=1}}^{n} \left(x-\zeta^s\right),
\end{align*}
is called the \emph{$n$-th cyclotomic polynomial} over $K$. When $(n,p)=1$, we have that $x^n-1=\prod_{d|n}Q_d(x)$, see \cite[Theorem 2.45]{RuNi97FiniteField}. 
Note that if $A^M=1$, for each irreducible factor $g$ of $f$, there exists an $m(g)$ such that 
$f|Q_{m(g)}(x)$. Thus we need to know how $Q_n(x)$ factors over $\F_q[x]$.
We recall the following result.
\begin{lemma}\cite[Theorem 2.47(ii)]{RuNi97FiniteField}\label{lem-cyclo-factor}
When $K=\F_q$ and $(n,q)=1$, then the cyclotomic polynomial $Q_n$ factors into $\phi(n)/e(n)$ distinct 
monic irreducible polynomials in $K[x]$ of the same degree $e(n)$, where $e(n)$ is order of $q$ in 
$\Z/n\Z^{\times}$.
\end{lemma}
Since $(M,p)=1$, the polynomial $x^M-1$ is separable and hence for $d\neq d'$ dividing
$M$, the cyclotomic polynomials $Q_d(x)$ and $Q_{d'}(x)$ are coprime. Hence we have the following lemma.
\begin{proposition}\label{lem:gen-sem-GL}
    Let $(M,q)=1$ and $\G_n=\GL_n(q)$. Let $a_n$ denote the number of 
    semisimple conjugacy classes
    $x_{n,i}^{G}$, such that $x_{n,i}^M=\id_n$ for some indices $i$. Then we have that
    \begin{equation}
        1+\sum\limits_{n=1}^{\infty} a_nz^n=\prod\limits_{\substack{d|M}}\left(\dfrac{1}{1-z^{e(d)}}\right)^{\frac{\phi(d)}{e(d)}},
    \end{equation}
    where $e(d)$ denotes the multiplicative order of $q$ in $\Z/d\Z^\times$.
\end{proposition}
\begin{proof}
    Let $A$ be a semisimple matrix satisfying $A^M=1$. Then the characteristic polynomial of 
    $A$, divides $x^M-1$. Since $(M,q)=1$, using factorization of $x^M-1$ and \cref{lem-cyclo-factor},
    we get that the irreducible factors of $x^M-1$ are of degree $\phi(d)/e(d)$ for $d|M$, where $e(d)$
    denotes the multiplicative order of $q$ in $\Z/d\Z^{\times}$. Furthermore, $(M,q)=1$ implies that
    the polynomial $x^M-1$ is separable. Hence
    for $d\neq d'$ dividing $M$, we have that $(Q_d(x),Q_{d'}(x))=1$. Hence the result follows.
\end{proof}
To deduce the result on the probability of a semisimple element to be an $M$-th root of identity, we need 
to know the order of the centralizer of each element. This has been discussed in \cref{sec:conjcent}. 
\textcolor{black}{\begin{corollary}\label{cr:semsp-GL-elements}
        Let $(M,q)=1$ and $\G_n=\GL_n(q)$. Let $b_n$ denote the proportion of
    semisimple $M$-th roots of identity in $\G_n$. Then we have that
    \begin{equation*}
        1+\sum\limits_{n=1}^{\infty} b_nz^n=\prod\limits_{\substack{d|M}}\left(1+
        \sum\limits_{m=1}^{\infty}\dfrac{z^{me(d)}}{|\GL_m(q^{e(d)})|}\right)^{\frac{\phi(d)}{e(d)}},
    \end{equation*}
    where $e(d)$ denotes the multiplicative order of $q$ in $\Z/d\Z^\times$.
\end{corollary}}
To calculate the number of semisimple conjugacy classes in $\Sp_{2n}(q)$, which are $M$-th root of $\id$, we need to know the factorization of $x^M-1$. Before proceeding further, for an integer $m$, 
let us define 
\begin{align*}
D_m=\{x\in \mathbb{Z}_{>0}: x|q^m+1,x\nmid q^{t}+1, \text{ for all }0\leq t<m\}.     
\end{align*}
The set $D_m$ plays a crucial role in counting the semisimple elements in finite symplectic and 
orthogonal groups, which gives $\id$ on raising the power to $M$. We collect the following results from \cite{YuMu04}.
\begin{lemma}\cite[Theorem 8, and Theorem 11]{YuMu04}\label{lem:self-reci-split}
We have the following results.
    \begin{enumerate}
        \item An irreducible polynomial of degree $2m$ over $\F_q$ is SRIM iff $\ord(f)\in D_m$
        \item For $d\in D_m$ the $d$-th cyclotomic polynomial $Q_d$ factors into product of all 
        distinct SRIM polynomials of degree $2m$ and order $d$.
    \end{enumerate}
\end{lemma}
\begin{corollary}\label{lem:star-factor}
The polynomial $Q_d(x)$ has a (and hence all) $*$-symmetric irreducible factor
if and only if $d\in D_m$ for some $m$. Furthermore, if $e(d)$ is odd, $Q_d$ has no SRIM factor.
\end{corollary}
\begin{proof}
    By definition, $Q_n$ is a $*$-symmetric polynomial. Hence its factors are either irreducible 
    $*$-symmetric or product of the form $gg^*$, where $g\neq g^*$ and $g$ are irreducible. If 
    $e(n)$ is odd, then there is no irreducible $*$-symmetric factor of $Q_n$,
    since all $*$-symmetric irreducible polynomials are of even degree. So we assume that
    $e(n)$ is even.
    So assume $e(n)$ is even. The backward direction is the content of \cref{lem:self-reci-split}. In the other direction,
    assume that $Q_n(x)$ has a $*$-symmetric polynomial, say $f$. Then $\ord(f)\in D_m$ for some $m$. But, by definition of $Q_n$, we have that $\ord(f)=n$ and hence $n\in D_m$.
\end{proof}
\begin{proposition}\label{lem:gen-sem-Sp}
    Let $(M,q)=1$ and $G=\Sp_{2n}(q)$. Let $a_n$ denote the number of
    semisimple conjugacy classes
    $x_{n,i}^{G}$, such that $x_{n,i}^M=\id_n$ for some indices $i$. Then we have that
    \begin{equation}
        1+\sum\limits_{n=1}^{\infty} a_nz^n=\prod\limits_{m=1}^{\infty}\left(\prod\limits_{\substack{d|M\\ d\in D_m}}\left(\dfrac{1}{1-z^{e(d)}}\right)^{\frac{\phi(d)}{e(d)}}\prod\limits_{\substack{d|M\\d\not\in D_m}}\left(\dfrac{1}{1-z^{2e(d)}}\right)^{\frac{\phi(d)}{2e(d)}}\right)
    \end{equation}
    where $e(d)$ denotes the multiplicative order of $q$ in $\Z/d\Z^\times$.
\end{proposition}
\begin{proof}
    Let $A\in\Sp_{2n}(q)$ be a matrix satisfying $A^M=\id$. Then we have that, the minimal polynomial of $A$ over $\F_q$
    should divide $x^M-1$. Since $x^M-1=\prod\limits_{d|M}Q_d(x)$ and $Q_d(x)$ further decomposes
    into irreducibles we need to do a case-by-case study. Because of \cref{lem:star-factor} we have to use 
    $D_m$'s. So first assume that $d\in D_m$, then each of the factors of $Q_d$ determine a Symplectic matrix of 
    $\Sp_{e(d)}$ which gives $\id$ when raised to the power $M$. Since $Q_d$ has in total 
    $\phi(d)/e(d)$ many irreducible factors, this justifies the first factor
    in the product. 

    Next, suppose $d\not\in D_m$. Then for any $g\neq g^*$, a factor of $Q_d$, we need to pair up $g$
    and $g^*$ to get a conjugacy class of $\Sp_{2e(d)}$, which contributes to the counting.
    Since $Q_d$ is $*$-symmetric, $g$ and $g^*$ occur together in the 
    factorization of $Q_d$. Furthermore, $Q_d$ is separable and hence all the irreducible factors
    are coprime to each other. Since $g$  and $g^*$ are combined together, the exponent of $z$ is taken 
    to be 
    $2e(d)$. But in the process, we are grouping $g$ and $g^*$ together, resulting in the power of 
    the factor being $\phi(d)/2e(d)$. This finishes the proof.
\end{proof}
\begin{corollary}\label{cr:semsp-Sp-elements}
        Let $(M,q)=1$ and $\G_n=\Sp_{2n}(q)$. Let $b_n$ denote the proportion of
    semisimple $M$-th roots of identity in $\G_n$. Then we have that 
    \begin{align*}
    \begin{split}
        1+\sum\limits_{n=1}^{\infty} b_nz^n&=\left(1+\sum\limits_{m=1}^{\infty}
       \dfrac{u^m}{|\Sp_{2m}(q)|}\right)^{o(M)}
       \prod\limits_{m=1}^{\infty}\left(\prod\limits_{\substack{d|M\\ d\in D_m}}\left(1+\sum\limits_{m=1}^{\infty}\dfrac{z^{me(d)}}{|\U_m(q^{2e(d)})|}\right)^{\frac{\phi(d)}{e(d)}}\right.\\
       &\times\left.
       \prod\limits_{\substack{d|M\\d\not\in D_m\\ d\neq 1,2}}\left(1+\sum\limits_{m=1}^{\infty}\dfrac{z^{me(d)}}{|\GL_m(q^{e(d)})|}\right)^{\frac{\phi(d)}{2e(d)}}\right)
    \end{split}  
    \end{align*}
    where $e(d)$ denotes the multiplicative order of $q$ in $\Z/d\Z^\times$, and $o(M)=2$ if $2|M$ and $1$ otherwise. 
\end{corollary}
\begin{proof}
    Start with the case when $M$ is even. In this case, both $\id$ and $\id$ 
    are $M$-th roots of identity. Now let us assume that $A$ is an $M$-th root 
    of identity, such that
    $\pm1$ is not an eigenvalue of $A$. Then the characteristic polynomial
    of $A$ is a product of factors of $Q_d$, where $d|M$ and the polynomials are irreducible
    of degree more than $1$. The rest of the proof follows from the description of the
    centralizers when the irreducible factor is a SRIM polynomial or not. In the latter case,
    we have that $g\neq g^*$ occurs together in the factorization of the 
    characteristic polynomial of $A$. For $M$ being even $-\id$ is not an $M$-th root of identity,
    which proves the existence of the function $o(M)$ in the formula. This finishes the proof.
\end{proof}
\begin{corollary}\label{cr:semsp-Or-elements}
        Let $(M,q)=1$ and $\G_{2n}=\On{\ep}_{2n}(q)$ and $\G_{2n+1}=\On{}_{2n+1}(q)$, where $\ep=\pm$. Let $b_n^{\ep}$ and $b_n$ denote the proportion of
    semisimple $M$-th roots of identity in $\G_{2n}$ and $\G_{2n+1}$ 
    respectively. 
    Let us define
    \begin{align*}
        B_{+}(z)=1+\sum\limits_{n=1}^{\infty} b_n^+z^n,
        B_{-}(z)=\sum\limits_{n=1}^{\infty} b_n^{-}z^n,
        B(z)=1+\sum\limits_{n=1}^{\infty}b_nz^n.
    \end{align*}    
    Then we have that
    \begin{align*}
        &B_+(z^2)+B_{-}(z^2)+2zB(z^2)\\
    =&\left(1+\sum\limits_{m=1}^{\infty}\left(\dfrac{1}{|\On{+}_{2m}(q)|}+\dfrac{1}{|\On{-}_{2m}(q)|}\right)z^{2m}+
    z\left(1+\sum\limits_{m=1}^{\infty}\dfrac{z^{2m}}{|\Sp_{2m}(q)|}\right)\right)^{o(M)}\\
       \times &\prod\limits_{m=1}^{\infty}\left(\prod\limits_{\substack{d|M\\ d\in D_m}}\left(1+\sum\limits_{m=1}^{\infty}\dfrac{z^{2me(d)}}{|\U_m(q^{2e(d)})|}\right)^{\frac{\phi(d)}{e(d)}}
       \prod\limits_{\substack{d|M\\d\not\in D_m\\ d\neq 1}}\left(1+\sum\limits_{m=1}^{\infty}\dfrac{z^{2me(d)}}{|\GL_m(q^{e(d)})|}\right)^{\frac{\phi(d)}{2e(d)}}\right)
    \end{align*}
    and
    \begin{align*}
        &B_{+}(z^2)-B_{-}(z^2)\\
        =&\left(1+\sum\limits_{m=1}^{\infty}\left(\dfrac{1}{|\On{+}_{2m}(q)|}-\dfrac{1}{|\On{-}_{2m}(q)|}\right)z^{2m}\right)^{o(M)}\\
        \times &\prod\limits_{m=1}^{\infty}\left(\prod\limits_{\substack{d|M\\ d\in D_m}}\left(1+\sum\limits_{m=1}^{\infty}\dfrac{(-1)^mz^{2me(d)}}{|\U_m(q^{2e(d)})|}\right)^{\frac{\phi(d)}{e(d)}}
       \prod\limits_{\substack{d|M\\d\not\in D_m\\ d\neq 1}}\left(1+\sum\limits_{m=1}^{\infty}\dfrac{z^{2me(d)}}{|\GL_m(q^{e(d)})|}\right)^{\frac{\phi(d)}{2e(d)}}\right)
    \end{align*}
    where $e(d)$ denotes the multiplicative order of $q$ in $\Z/d\Z^\times$, and $o(M)=2$ if $2|M$ and $1$ otherwise. 
\end{corollary}
\begin{proof}
    The proof is almost similar to \cref{cr:semsp-Sp-elements}, but care needs to be 
    taken as we are giving it in a combinatorial formula. This formula occurs 
    because of the existence of different quadratic forms. Let $X$ be an orthogonal $M$-th root of identity, preserving a non-degenerate 
    quadratic form on an $N$ dimensional vector space $V$, where $N$ is even or odd.
    Consider the characteristic polynomial of $X$, which decomposes into irreducibles
    with factors being $\varphi=\varphi^*$ and irreducible or $\psi\psi^*$, with $\psi\neq\psi^*$ irreducible. We use the centralizer information from
    \cref{sec:conjcent}. 
    
    Let us start with the first equation. The characteristic polynomial of $X$ might have root to be
    $1$ or $-1$. But, if $M$ is odd, it will not have $-1$ as a root. This justifies the presence of $o(M)$. Note that for $N$ being even and positive
    the coefficient of $z^N$ is sum of the probabilities in case of $\On{+}_N(q)$
    and $\On{-}_{N}(q)$. Furthermore for odd $N$ this is two times the probability in the case of $\On{}_N(q)$, which occurs because of two quadratic forms.
    In case of odd $N$, note that $|\On{}_N(q)|=2|\Sp_{N-1}(q)|$, which we are using in the formula. This proves the equality in the first equation.

    We prove the next equation by modifying the first equation. We need the value of
    type of an orthogonal space, as was discussed in \cref{sec:conjcent}. This consideration doesn't alter the first factor, as they correspond to the case whether 
    $\id$ and $-\id$ are $M$-th roots of identity or not.
    In the right side of the first equation, for the self conjugate irreducible factors,
    replace $z$ by $\pm i z$, according to their types. This proves the second equality. 
\end{proof}
Let us now deduce the SCIM factors of $x^M-1$ in $\F_{q^2}$. According to
\cite[Lemma 2]{Enn62}, if $f$ is a SCIM polynomial of odd degree $d$, then for any root
$\zeta$ of $f$, it satisfies $\zeta^{q^d+1}=1$. This means that $\ord(f)|q^d+1$.
For a positive integer $m$ define
\begin{align*}
    \D_m=\{y\in\mathbb{Z}_{>0}:y|q^{2m+1}+1, y\nmid q^t+1\text{ for odd }1\leq t<2m+1\}.
\end{align*}
It is evident that if $f$ is a SCIM polynomial of degree $d$, then $\ord(f)\in \D_{(d-1)/2}$. 
If $d\in D_m$, and $\beta $ is a primitive $d$-th root of unity, then the set 
$\{\beta, \beta^{q^2}, \beta^{q^4}, \cdots, \beta^{q^{2(d-1)}}\}$ is of cardinality $d$. Indeed,
if $\beta^{q^{2i}}=\beta^{q^{2j}}$ for some $0\leq i<j \leq d-1$, then $d|q^{2(j-i)}-1$. Hence 
$2(j-i)$ is an even multiple of $d$, which forces $2(j-i)=0$, see \cite[Prposition 1]{YuMu04}.
For $e\in \D_m$ and a primitive $e$-th root of unity, let us define
\begin{align*}
    f_{\beta}(x)=\prod\limits_{i=0}^{d-1}\left(x-\beta^{q^{2i}}\right).
\end{align*}
Then by \cite[Theorem 3.4.8]{LingXing04} this is an irreducible polynomial. We claim that $f_\beta$ is 
a SCIM polynomial. Note that
 $\left(\beta^{q^{2j}}\right)^{-q}=\beta^{q^{2(m+i+1)}}$, since $\beta^{q^{2m+1}+1}=1$.
 This implies that $\left(\beta^{q^{2j}}\right)^{-q}$ is also a root, and hence
$f_{\beta}$ is a SCIM polynomial, see \cite[pp. 23]{FuNePr05}. Using \cite[Theorem 2.47]{RuNi97FiniteField}, we conclude that for an irreducible polynomial $f$ of degree $2d+1$ over $\F_{q^2}$, the following
statements are equivalent;
\begin{enumerate}
    \item $f$ is SCIM,
    \item $\ord(f)\in \D_{d}$,
    \item $f=\widetilde{f}_{\beta}$ for some $\beta$ being a $t$-th root of unity for some $t\in \D_{d}$.
\end{enumerate}
For $t\in \D_{d}$, we have that $Q_d(x)=\prod(x-\gamma)$, where $\gamma$ runs over all $t$-th primitive 
root of unity. It follows that in this case $Q_d$ factors into a product of all SCIM polynomials. Using a
similar argument as in \cref{lem:star-factor} we conclude the following.
\begin{lemma}\label{lem:tilde-factor}
The polynomial $Q_d(x)$ has a (and hence all) $\sim$-symmetric irreducible factor
if and only if $d\in \D_m$ for some $m$. Furthermore, if $e(d)$ is even, $Q_d$ has no SCIM factor.
\end{lemma}
\begin{proposition}\label{lem:gen-sem-Un}
    Let $(M,q)=1$ and $G=\U_{n}(q^2)$. Let $a_n$ denote the number of
    semisimple conjugacy classes
    $x_{n,i}^{G}$, such that $x_{n,i}^M=\id_n$ for some indices $i$. Then we have that
    \begin{equation}
        1+\sum\limits_{n=1}^{\infty} a_nz^n=\prod\limits_{m=1}^{\infty}\left(\prod\limits_{\substack{d|M\\ d\in \D_m}}\left(\dfrac{1}{1-z^{e(d)}}\right)^{\frac{\phi(d)}{e(d)}}\prod\limits_{\substack{d|M\\d\not\in \D_m}}\left(\dfrac{1}{1-z^{2e(d)}}\right)^{\frac{\phi(d)}{2e(d)}}\right)
    \end{equation}
    where $e(d)$ denotes the multiplicative order of $q$ in $\Z/d\Z^\times$.
\end{proposition}
\begin{proof}
    Let $A\in\U_{n}(q^2)$ be a matrix satisfying $A^M=\id$. Then we have that, the minimal polynomial of $A$ over $\F_{q^2}$
    should divide $x^M-1$. Since $x^M-1=\prod\limits_{d|M}Q_d(x)$ and $Q_d(x)$ further decomposes
    into irreducibles we examine this each $d$ case-by-case depending on whether
    $d$ is inside or outside $\D_m$ for some $m$, and use \cref{lem:tilde-factor}. So first assume that $d\in \D_m$, then each of the factors of $Q_d$ determine a Unitary matrix of 
    $\U_{e(d)}(q^2)$ which gives $\id$ when raised to the power $M$. Since $Q_d$ has in total 
    $\phi(d)/e(d)$ many irreducible factors, this justifies the first factor
    in the product. 

    Next, suppose $d\not\in D_m$. Then for any $g\neq \widetilde{g}$, a factor of $Q_d$, we need to pair up $g$
    and $g^*$ to get a conjugacy class of $\U_{e(d)}(q^2)$, which contributes to the counting.
    Since $Q_d$ is $\sim$-symmetric (because $\alpha$ and $\alpha^{-q}$ are roots of $Q_d$, as $(d,M)=1$), $g$ and $\widetilde{g}$ occur together in the 
    factorization of $Q_d$. Furthermore, $Q_d$ is separable and hence all the irreducible factors
    are coprime to each other. Hence the exponent of $z$ is taken 
    to be 
    $2e(d)$. But in the process, we are grouping $g$ and $g^*$, resulting in the power of 
    the factor being $\phi(d)/2e(d)$. This finishes the proof.
\end{proof}
\begin{corollary}\label{cr:semsp-U-elements}
        Let $(M,q)=1$ and $\G_n=\U_n(q^2)$. Let $b_n$ denote the proportion of
    semisimple $M$-th roots of identity in $\G_n$. Then we have $1+\sum\limits_{n=1}^{\infty} b_nz^n$ to be equal to
    \begin{equation}
       \prod\limits_{m=1}^{\infty}\left(\prod\limits_{\substack{d|M\\ d\in \D_m}}\left(1+\sum\limits_{m=1}^{\infty}\dfrac{z^{me(d)}}{|\U_m(q^{2e(d)})|}\right)^{\frac{\phi(d)}{e(d)}}\prod\limits_{\substack{d|M\\d\not\in \D_m}}\left(1+\sum\limits_{m=1}^{\infty}\dfrac{z^{2me(d)}}{|\GL_m(q^{2e(d)})|}\right)^{\frac{\phi(d)}{2e(d)}}\right)
    \end{equation}
    where $e(d)$ denotes the multiplicative order of $q$ in $\Z/d\Z^\times$.
\end{corollary}
\begin{proof}
    This follows easily from the information on centralizers and size of the conjugacy classes.
    The flow of reasoning is as same as \cref{cr:semsp-GL-elements},
     \cref{cr:semsp-Sp-elements} and \cref{cr:semsp-Or-elements}. Hence it is left to the reader. 
\end{proof}
\subsection{Unipotent classes} Recall that an element $x\in G(q)$ is a unipotent element
if and only if $\ord(x)$ is a power of the defining prime.
Hence to find all unipotent elements of order $M$, it is enough to find elements 
of order $p^r$ where $M=t\cdot p^r$ and $p\nmid t$. In the following result, we 
find the conjugacy classes of unipotent elements of order $p^r$ when $G=\GL_n$. 
In the other finite groups of Lie type two unipotent elements $u_1$ and $u_2$
need not be $G$-conjugate, even if they are $\GL_n$ conjugate. But it so happens that
the unipotent conjugacy class in other Finite groups of Lie type descend 
from those in $\GL_n$. Thus it is necessary that we first resolve the
case of $\GL_n(q)$. The treatments for the other groups will follow easily.
\begin{proposition}\label{prop:order-of-unipotent}
    Let $\lambda=\lambda_1^{i_1}\lambda_2^{i_2}\ldots\lambda_k^{i_k}$ be a partition of 
    $|\lambda|$, where $\lambda_1>\lambda_2>\ldots>\lambda_k$. 
    Consider $J_{\lambda}$ to be the matrix corresponding to the Jordan 
    canonical form of the unipotent matrix attached with $\lambda$. Then 
    $J_{\lambda}$ is of order $p^r$ if and only if $p^{r-1}<\lambda_1\leq p^r$.
\end{proposition}
\begin{proof}
    Note that $N_{\lambda}=J_{\lambda}-\id=(a_{ij})$ is a nilpotent matrix with the $(i,i+1)$-th entry $a_{ii+1}\in\{0,1\}$ for all 
    $1\leq i\leq |\lambda|-1$ and $a_{ij}=0$ for all $(i,j)\neq (i,i+1)$.
    Also $J_{\lambda}-\id$ is a nilpotent matrix satisfies $(J_{\lambda}-\id)^{|\lambda|}=0$. 
    Given that $J_{\lambda}\in\GL_{|\lambda|}(\F{p})$ is a nontrivial unipotent element, it must have order
    $p^t$ for some $t>0$. Using the binomial theorem,
    \begin{align*}
        J_\lambda^{p^t}=\id+\displaystyle\sum\limits_{i=1}^{p^t}{p^t \choose i}N_{\lambda}^i.
    \end{align*}
    Since $\lambda=\ell_1^{i_1}\ell_2^{i_2}\ldots\ell_k^{i_k}$, the $t$ should satisfy $p^t\geq \ell_1$. This proves both sides of the statement.
\end{proof}
\begin{lemma}\label{lem:}
    Let $Q(r,[s])$ denote the number of partitions of $r$, with parts not exceeding $s$. Then the generating function for $Q$ in indeterminate $z$ is given by
    \begin{align*}
        1+\sum\limits_{r=1}^{\infty} Q(r,[s])z^r=\prod\limits_{t=1}^{s}\dfrac{1}{1-z^t}.
    \end{align*}
\end{lemma}

\section{Generating functions for all roots of identity}\label{sec:gen-fun}
To detect an element  $x\in \G(q)$, such that $x^M=\id$, we use the Jordan decomposition. 
First, let us write $M=t\cdot p^r$, where $p\nmid t$. Using Jordan 
decomposition (see \cite{MaTest11}), we write $x=x_sx_u$, where $x_s$ is semisimple
and $x_u$ is unipotent. It can be easily seen that $x^M=\id$ if and only if 
$x_s^t=\id$ and $x_u^{p^r}=\id$. Since we use the combinatorial data to conclude 
our results,
it is important to know the unipotent part of $x$, when $x$ is determined by a 
combinatorial data $\Delta(x)$. For any element $x\in \G(q)$, we consider 
it inside ambient $\GL_n(q)$, and since the order of an element is unaltered under this consideration, we 
discuss the case for 
an element in $\GL_n$ only. The other cases will follow easily.
Recall from \cref{sec:conjcent}, for an element $x\in \GL_n(q)$, if the combinatorial data
is given by $\{(f,\lambda_f(x)):f\in\Phi\setminus\{z\}\}$, then the companion matrix corresponding to the combinatorial data is given by the sum of the blocks of the form
\begin{align*}
    \begin{pmatrix}
        C_f & \id & & \\
            & C_f & &\\
            &&\ddots &\\
            &&&C_f
    \end{pmatrix},
\end{align*}
where $C_f$ is the companion matrix corresponding to $f$, and $\id$ is a $\deg f\times \deg f$ matrix. Also, the size of this blocks are $\lambda_i\deg(f)\times \lambda_i\deg(f)$, where $\lambda=(\lambda_1,\lambda_2,\ldots)$.
It is well known that two matrices $x,x'\in\GL_n$ are conjugate to each other if
and only if they are conjugate in $\GL_n(\overline{\F_q})$, see \cite{BuGi2016primes}. Hence we consider the matrix $x$ in $\GL_n(\overline{\F_q})$, up to conjugacy. Then the combinatorial data $(f,\lambda_f(x))$ becomes
$\{(x-\mu_1,\lambda_f(x)),\ldots,(x-\mu_s,\lambda_f(x))\}$, where $f(x)=(x-\mu_1)\ldots(x-\mu_s)$.
Hence to summarize we have the following result, using \cref{prop:order-of-unipotent}.
\begin{lemma}
    Let $\Delta(x)=\{(f,\lambda_f(x)):f\in \Phi\}$ be the combinatorial data 
    corresponding to $x$ and $M=t\cdot p^r$, $p\nmid t$. If $x=x_sx_u$ is the Jordan decomposition of $x$, then $x$ is an 
    $M$-th root of identity in $\GL_n(q)$ if and only if $x_s$ is a $t$-th root of
    identity in $\GL_n(q)$ and each part of $\lambda(f)$ has size less than
    or equal to $p^r$.
\end{lemma}

In this section, we use cycle indices for classical groups 
to conclude the final set of results. We will recall the cycle indices from \cite{Ful99CycInd} and \cite{Ful02randomMatrix} whenever needed 
and then deduce our results. We assume the polynomials to be in the indeterminate $u$. Let $x_{\varphi,\lambda}$ be variables corresponding to pairs of polynomials and
partitions. The cycle index for a group $\G(q)$ is defined to be
\begin{align*}
    1+\sum\limits_{n=1}^{\infty}\dfrac{z^n}{|\G_n(q)|}\left(\sum\limits_{\alpha\in \G_n(q)}
    \prod\limits_{\substack{\varphi\neq u}}x_{\varphi,\lambda_{\varphi}(\alpha)}\right).
\end{align*}
In the next couple of results, we are first going to recall the factorization of the
cycle indices, due to the works of Jason Fulman. After that, we are going to
substitute zero or one, on the basis of whether such an element contributes to the
quantity or not. We present the first lemma now.
\begin{lemma}\cite[pp. 55]{Ful02randomMatrix}\label{lem:cyc-ind-GL}
    For $G(q)=\GL_n(q)$, we have 
    \begin{align*}
        &1+\sum\limits_{n=1}^{\infty}\dfrac{z^n}{|\GL_n(q)|}\left(\sum\limits_{\alpha\in \GL_n(q)}
    \prod\limits_{\substack{\varphi\neq u}}x_{\varphi,\lambda_{\varphi}(\alpha)}\right)\\
    =&\prod\limits_{\varphi\neq u}\left(1+\sum\limits_{n\geq 1}\sum\limits_{\lambda\vdash n}
    x_{\varphi,\lambda}
    \dfrac{z^{n\cdot\deg\varphi}}{q^{\deg\varphi\cdot(\sum_{i}(\lambda_i')^2)}\prod\limits_{i\geq 1}\left(\dfrac{1}{q^{\deg\varphi}}\right)_{m_i(\lambda_{\varphi})}}\right).
    \end{align*}
\end{lemma}
\begin{theorem}\label{th:GL-all-probability}
    Let $a_n$ denote the number of elements in $\GL_n(q)$ which are $M$-th root of identity. Let 
    $M=t\cdot p^r$, where $p\nmid t$. Then the generating function of the probability $a_n/|\GL_n(q)|$ is given by
    \begin{align*}
    &1+\sum\limits_{n=1}^{\infty}\dfrac{a_n}{|\GL_n(q)|}z^n=
        \prod\limits_{d|t}\left(1+\sum\limits_{m\geq 1}\sum\limits_{\substack{\lambda\vdash m\\\lambda_1\leq p^r}}
    \dfrac{z^{me(d)}}{q^{e(d)\cdot(\sum_{i}(\lambda_i')^2)}\prod\limits_{i\geq 1}\left(\dfrac{1}{q^{e(d)}}\right)_{m_i(\lambda_{\varphi})}}\right)^{\frac{\phi(d)}{e(d)}},
    \end{align*}
    where $e(d)$ denotes the multiplicative order of $q$ in $\Z/d\Z^\times$.
\end{theorem}
\begin{proof}
It follows from the proof of \cref{lem:gen-sem-GL} and \cref{prop:order-of-unipotent} that, for an element $x\in \GL_n(q)$, 
\begin{enumerate}
    \item the semisimple part $x_s$ has order $t$ if and only if the irreducible 
    factors of the characteristic polynomial of $x_s$ divides $Q_d$ for some $d|t$,
    \item the unipotent part $x_u$ has order $p^r$ if and only if the partition 
    corresponding to $x_u$ has all parts lesser than or equal to $p^r$.
\end{enumerate}
Hence in the formula of \cref{lem:cyc-ind-GL}, we substitute $x_{\varphi,\lambda_{\varphi}(\alpha)}$ to be $1$ when
\begin{enumerate}
    \item all the polynomials occurring in $\Delta(\alpha)$ are divisors of $Q_d$,
    for some $d|t$,
    \item all the $\lambda_{\varphi}(\alpha)$ occurring in $\Delta(\alpha)$ has 
    highest part to be lesser than or equal to $p^r$.
\end{enumerate}
and put all other $x_{\varphi,\lambda_{\varphi}(\alpha)}$ to be zero. The occurrence of the degrees follows easily. This proves the equality among both sides.
\end{proof}
The theorems for the case of finite symplectic, orthogonal and unitary groups will
be stated without detailed proof, since the arguments will be as same as
\cref{th:GL-all-probability}, so we won't be repeating them. However, we will indicate the results from which the proofs follow. 
\begin{lemma}\cite{Ful99CycInd}\label{lem:cyc-ind-sp}
    For $G(q)=\Sp_{2n}(q)$, we have
    \begin{align*}
        &1+\sum\limits_{n=1}^{\infty}\dfrac{z^{2n}}{|\Sp_{2n}(q)|}\left(\sum\limits_{\alpha\in\Sp_{2n}(q)}
        \prod\limits_{\varphi=u\pm 1}x_{\varphi, \lambda^{\pm}_\varphi(\alpha)}
        \prod\limits_{\varphi\neq u\pm 1} x_{\varphi,\lambda_\varphi(\alpha)}
        \right)\\
        =&
        \prod\limits_{\varphi=u\pm 1}\left(1+\sum\limits_{n\geq 1}\sum\limits_{\lambda^{\pm}\vdash n} x_{\varphi,\lambda^{\pm}}\dfrac{z^{n}}{c_{\Sp,u\pm1,q}(\lambda^{\pm})}\right)
       \prod\limits_{\substack{\varphi=\varphi^*\\\varphi\neq u\pm 1}}\left(1+\sum\limits_{n\geq 1}\sum\limits_{\lambda\vdash n} x_{\varphi,\lambda}\dfrac{(-z^{\deg\varphi})^{n}}{c_{\GL,u-1,-(q^{\deg\varphi})^{1/2}}(\lambda)}\right)\\
       \times & \prod\limits_{\substack{\{\varphi,\varphi^*\}\\\varphi\neq \varphi^*}}\left(1+\sum\limits_{n\geq 1}\sum\limits_{\lambda\vdash n} x_{\varphi,\lambda}x_{\varphi^*,\lambda}\dfrac{z^{2n\deg\varphi}}{c_{\GL,u-1,q^{\deg\varphi}}(\lambda)}\right)
    \end{align*}
\end{lemma}
\begin{theorem}\label{th:Sp-all-probability}
    Let $a_n$ denote the number of elements in $\Sp_{2n}(q)$ which are $M$-th root of identity. Let 
    $M=t\cdot p^r$, where $p\nmid t$. Then the generating function of the probability $a_n/|\Sp_{2n}(q)|$ is given by
    \begin{align*}
        &1+\sum\limits_{n=1}^{\infty} \dfrac{a_n}{|\Sp_{2n}(q)|}z^n\\
        =&\left(1+\sum\limits_{n\geq 1}\sum\limits_{\substack{\lambda^{\pm}\vdash n\\\lambda^\pm_1\leq p^r}} \dfrac{z^{n}}{c_{\Sp,u\pm1,q}(\lambda^{\pm})}\right)^{o(t)}
        \prod\limits_{m=1}^{\infty}\left(
        \prod\limits_{\substack{d|t\\d\in D_m}}\left[1+\sum\limits_{n\geq 1} \sum\limits_{\substack{\lambda}}\dfrac{(-z)^ne(d)}{c_{\GL,u-1,-(q^{e(d)})^{1/2}}(\lambda)}\right]^{\frac{\phi(d)}{e(d)}}\right.\\
        &\left.
        \prod\limits_{\substack{d|t\\d\not\in D_m, d\neq 1, 2}}\left[1+\sum\limits_{n\geq 1}\sum\limits_{\lambda\vdash n}\dfrac{z^{2ne(d)}}{c_{\GL,u-1,{q^{e(d)}}}(\lambda)}\right]^{\frac{\phi(d)}{2e(d)}}
        \right)
    \end{align*}
\end{theorem}
\begin{proof}
    This follows from the proof of \cref{cr:semsp-Sp-elements} and \cref{prop:order-of-unipotent}. The occurrence of $o(t)$ can be justified 
    as was done in \cref{cr:semsp-Sp-elements}. Finally, we need to use \cref{lem:cyc-ind-sp}.
\end{proof}
\begin{lemma}\cite{Ful99CycInd}\label{lem:cyc-ind-O}
    Define the cycle index for sum of the both type of orthogonal groups to be
    \begin{align*}
        1+\sum_{n=1}^{\infty}&\left(\dfrac{z^n}{|\On{+}_n(q)|}\sum\limits_{\alpha\in\On{+}(q)}\prod\limits_{\varphi=u\pm 1} x_{\varphi,\lambda_\varphi^{\pm}(\alpha)}\prod\limits_{\varphi\neq u,u\pm 1}x_{\varphi,\lambda_\varphi(\alpha)}\right.\\
        +&\left.\dfrac{z^n}{|\On{-}_n(q)|}\sum\limits_{\alpha\in\On{-}(q)}\prod\limits_{\varphi=u\pm 1} x_{\varphi,\lambda_\varphi^{\pm}(\alpha)}\prod\limits_{\varphi\neq u,u\pm 1}x_{\varphi,\lambda_\varphi(\alpha)}\right).
    \end{align*}
    This quantity factorises as
    \begin{align*}
        &\prod\limits_{\varphi=u\pm 1}\left(1+\sum\limits_{n\geq 1}\sum\limits_{\lambda^{\pm}\vdash n} x_{\varphi,\lambda^{\pm}}\dfrac{z^{n}}{c_{\On{},u\pm1,q}(\lambda^{\pm})}\right)
       \prod\limits_{\substack{\varphi=\varphi^*\\\varphi\neq u\pm 1}}\left(1+\sum\limits_{n\geq 1}\sum\limits_{\lambda\vdash n} x_{\varphi,\lambda}\dfrac{(-z^{\deg\varphi})^{n}}{c_{\GL,u-1,-(q^{\deg\varphi})^{1/2}}(\lambda)}\right)\\
       \times & \prod\limits_{\substack{\{\varphi,\varphi^*\}\\\varphi\neq \varphi^*}}\left(1+\sum\limits_{n\geq 1}\sum\limits_{\lambda\vdash n} x_{\varphi,\lambda}x_{\varphi^*,\lambda}\dfrac{z^{2n\deg\varphi}}{c_{\GL,u-1,q^{\deg\varphi}}(\lambda)}\right)
    \end{align*}
\end{lemma}
\begin{theorem}\label{th:Or-all-probability}
    Let $a^{\epsilon}_n$ denote the number of elements in $\On{\epsilon}_{n}(q)$ which are $M$-th root of identity, where $\epsilon\in\{\pm\}$. Let 
    $M=t\cdot p^r$, where $p\nmid t$. Then the generating function of the sum of the probabilities $a^\epsilon_n/|\On{\epsilon}_{n}(q)|$ is given by
    \begin{align*}
        &1+\sum\limits_{n=1}^{\infty} \left(\dfrac{a_n^+}{|\On{+}_{n}(q)|}+\dfrac{a_n^-}{|\On{-}_{n}(q)|}\right)z^n\\
        =&\left(1+\sum\limits_{n\geq 1}\sum\limits_{\substack{\lambda^{\pm}\vdash n\\\lambda^\pm_1\leq p^r}} \dfrac{z^{n}}{c_{\On{},u\pm1,q}(\lambda^{\pm})}\right)^{o(t)}
        \prod\limits_{m=1}^{\infty}\left(
        \prod\limits_{\substack{d|t\\d\in D_m}}\left[1+\sum\limits_{n\geq 1} \sum\limits_{\substack{\lambda}}\dfrac{(-z)^ne(d)}{c_{\GL,u-1,-(q^{e(d)})^{1/2}}(\lambda)}\right]^{\frac{\phi(d)}{e(d)}}\right.\\
        &\left.
        \prod\limits_{\substack{d|t\\d\not\in D_m, d\neq 1, 2}}\left[1+\sum\limits_{n\geq 1}\sum\limits_{\lambda\vdash n}\dfrac{z^{2ne(d)}}{c_{\GL,u-1,{q^{e(d)}}}(\lambda)}\right]^{\frac{\phi(d)}{2e(d)}}
        \right)
    \end{align*}
\end{theorem}
\begin{proof}
    The reason for clubbing these two probabilities is related to 
    how the cycle indices for the different types of orthogonal groups are treated.
    This follows from the proof of \cref{cr:semsp-Or-elements} and \cref{prop:order-of-unipotent}. The occurrence of $o(M)$ can be justified 
    as was done in \cref{cr:semsp-Or-elements}. Finally, we need to use \cref{lem:cyc-ind-O}.
\end{proof}
A formula for the generating function for the difference of the probabilities
$\dfrac{a^+_n}{|\On{+}_n(q)|}-\dfrac{a^-_n}{|\On{-}(q)|}$ can be formulated easily, and using the same techniques we can obtain a formula for the generating function. This is omitted here. But, a treatment for a special case can be found in \cite{FuNePr05}. 
\begin{lemma}\cite[Theorem 10, pp 63.]{Ful99CycInd, Ful02randomMatrix}\label{lem:cyc-ind-Un}
     For $G(q)=\U_n(q^2)$, we have 
    \begin{align*}
        &1+\sum\limits_{n=1}^{\infty}\dfrac{z^n}{|\U_n(q)|}\left(\sum\limits_{\alpha\in \U_n(q^2)}
    \prod\limits_{\substack{\varphi\neq u}}x_{\varphi,\lambda_{\varphi}(\alpha)}\right)\\
    =&\prod\limits_{\substack{\varphi\neq u\\\varphi=\widetilde{\varphi}}}\left[1+\sum\limits_{n\geq 1}\sum\limits_{\lambda\vdash n}
    x_{\varphi,\lambda}
    \dfrac{(-z)^{n\cdot\deg\varphi}}{(-q)^{\deg\varphi\cdot(\sum_{i}(\lambda_i')^2)}\prod\limits_{i\geq 1}\left(\dfrac{1}{(-q)^{\deg\varphi}}\right)_{m_i(\lambda_{\varphi})}}\right]
    \\
    \times
    &
    \prod\limits_{\substack{\varphi\neq \widetilde{\varphi}\\\{\varphi, \widetilde{\varphi}\}}}\left[1+\sum\limits_{n\geq 1}\sum\limits_{\lambda\vdash n}
    x_{\varphi,\lambda} x_{\widetilde{\varphi},\lambda}
    \dfrac{z^{2n\cdot\deg\varphi}}{q^{2\deg\varphi\cdot(\sum_{i}(\lambda_i')^2)}\prod\limits_{i\geq 1}\left(\dfrac{1}{q^{2\deg\varphi}}\right)_{m_i(\lambda_{\varphi})}}\right]
    \end{align*}
\end{lemma}
\begin{theorem}\label{th:Un-all-probability}
    Let $a_n$ denote the number of elements in $\U_n(q^2)$ which are $M$-th root of identity. Let 
    $M=t\cdot p^r$, where $p\nmid t$. Then the generating function for the probability $a_n/|\U_n(q^2)|$ is given by
    \begin{align*}
    1+\sum\limits_{n=1}^{\infty}\dfrac{a_n}{|\U_n(q^2)|}z^n
    =&\prod\limits_{m=1}^{\infty}\left(\prod\limits_{\substack{d|t\\d\in\widetilde{D}_m}}\left[1+\sum\limits_{n\geq 1}\sum\limits_{\substack{\lambda\vdash n\\\lambda_1\leq p^r}}
    \dfrac{(-z)^{n\cdot e(d)}}{(-q)^{e(d)\cdot(\sum_{i}(\lambda_i')^2)}\prod\limits_{i\geq 1}\left(\dfrac{1}{(-q)^{e(d)}}\right)_{m_i(\lambda_{\varphi})}}\right]^{\frac{\phi(d)}{e(d)}}
    \right.\\\times&\left.
    \prod\limits_{\substack{d|t\\ d\not\in\widetilde{D}_m}}\left[1+\sum\limits_{n\geq 1}\sum\limits_{\substack{\lambda\vdash n\\ \lambda_1\leq p^r}}
    \dfrac{z^{2n\cdot e(d)}}{q^{2e(d)\cdot(\sum_{i}(\lambda_i')^2)}\prod\limits_{i\geq 1}\left(\dfrac{1}{q^{2e(d)}}\right)_{m_i(\lambda_{\varphi})}}\right]^{\frac{\phi(d)}{2e(d)}}\right),
    \end{align*}
    where $e(d)$ denotes the multiplicative order of $q$ in $\Z/d\Z^\times$.
\end{theorem}
\begin{proof}
    This follows from the proof of \cref{cr:semsp-U-elements} and \cref{prop:order-of-unipotent}. Finally, we need to use \cref{lem:cyc-ind-Un}.
\end{proof}
\section{An example}\label{sec:calc-prime}
{In this section, we compute the exact probility when $q\equiv -1\pmod{M}$}. All the groups 
are defined over $\fq$.
We concentrate mainly on the case when $M\not=2$ is a prime and $q\equiv-1\pmod{M}$, for example, $(q, M)=(41,7)$. Then
we get
\begin{align*}
    x^M-1=(x-1)\cdot Q_M(x).
\end{align*}
Using \cref{lem-cyclo-factor}, it will factor into $(M-1)/e(M)$ many distinct monic polynomials of degree $e(M)=2$.
In this case, the elements contributing to $M$-th root of identity will all be semisimple (see the discussion at the beginning of \cref{sec:semsim-uni}).
\subsection{The case of $\GL_n(q)$} Let $b_n$ denote the proportion of $M$-th roots of identity in $\GL_n(q)$. Then using \cref{cr:semsp-GL-elements}, we get that 
\begin{align*}
    1+\sum\limits_{n=1}^{\infty}b_nz^n&=\left(1+\sum\limits_{m=1}^{\infty}\dfrac{z^m}{|\GL_m(q)|}\right)\left(1+\sum\limits_{m=1}^{\infty}\dfrac{z^{2m}}{|\GL_m(q^2)|}\right)^{{(M-1)}/{2}}.
\end{align*}
We now divide the computation into two cases. The first is $n$ being odd. In this 
case, we should have an odd power of $z$, coming from the first term of the product. Other contributing powers of $z$ will have all even power. Hence probability of being an $M$-th root is
{
\begin{align*}
    \sum\limits_{\substack{1\leq j \leq M\\j=\text{odd}}}\dfrac{1}{|\GL_j(q)|}\cdot\left(\sum\limits_{\lambda\vdash \frac{M-j}{2}}\prod\limits_{\ell}\dfrac{1}{|\GL_{\lambda_{\ell}}(q^2)|}\right),
\end{align*}
where $\ell$ runs over the subscripts of the parts of $\lambda=(\lambda_1,\lambda_2,\ldots)$.
When $n$ is even, using the same argument as before, we get the resulting probability to be
\begin{align*}
    \sum\limits_{\substack{0\leq j \leq M\\j=\text{even}}}\dfrac{1}{|\GL_j(q)|}\cdot\left(\sum\limits_{\lambda\vdash \frac{M-j}{2}}\prod\limits_{\ell}\dfrac{1}{|\GL_{\lambda_{\ell}}(q^2)|}\right),
\end{align*}
where $\ell$ runs over the subscripts of the parts of $\lambda=(\lambda_1,\lambda_2,\ldots)$ and $|\GL_0(q)|$ is $1$ by convention.}
\subsection{The case of $\Sp_{2n}(q)$} Since $M$ is odd, we have that $o(M)=1$. We can have two possibilities, either $M\in D_m$ for some 
$m$, or $M\not\in D_m$ for all $m\geq 1$. If $M\in D_m$ for some $m$, then we have that the probability of an element of $\Sp_{2n}(q)$
to be $M$-th root, using \cref{cr:semsp-Sp-elements}, is
\begin{align*}
    \sum\limits_{j=1}^{n}\dfrac{1}{|\Sp_{2j}(q)|}\cdot \left( \sum\limits_{\lambda\in \p_{\frac{M-1}{2}}({M-j})}\prod\limits_{\ell}\dfrac{1}{|\U_{\lambda_\ell}(q^4)|}\right),
\end{align*}
and when $M\not\in D_m$, then the resulting probability will be
\begin{align*}
    \sum\limits_{j=1}^{n}\dfrac{1}{|\Sp_{2j}(q)|}\cdot \left( \sum\limits_{\lambda\in \p_{\frac{M-1}{4}}({M-j})}\prod\limits_{\ell}\dfrac{1}{|\GL_{\lambda_\ell}(q^2)|}\right).
\end{align*}

The cases for the orthogonal groups $\On{\epsilon}_n(q)$ and the unitary group $\U_n(q^2)$ are similar and we omit them from the display.
\printbibliography
\end{document}